\documentclass[11pt,english,a4paper]{smfart}
\usepackage[english]{babel}
\usepackage{amsmath, amsthm, amscd, amsfonts, amssymb, graphicx, xcolor, setspace}
\usepackage[pagebackref,colorlinks=true,linkcolor=blue,citecolor=blue,urlcolor=red, hypertexnames=true]{hyperref}
\usepackage[abbrev,backrefs,nobysame]{amsrefs}
 \usepackage{enumerate}
\usepackage{amsmath,amssymb,bm,amsfonts}
\usepackage{bbm}
\usepackage[all]{xy}
\entrymodifiers={+!!<0pt,\fontdimen22\textfont2>}

\usepackage{mathrsfs}

\newtheorem{theorem}{Theorem}[section]

\newtheorem{proposition}[theorem]{Proposition}

\newcommand{\fl}[2]{
\xymatrix@C15pt{#1\ar[r]&#2}}
\newcommand{\flcourte}[2]{
\xymatrix@C12pt{#1\ar[r]&#2}}

{\theoremstyle{definition}}

{\theoremstyle{definition}}

\theoremstyle{definition}

\newtheorem{question}[theorem]{Question}

{\theoremstyle{definition}\newtheorem{remark}[theorem]{Remark}}

\def\N{\ensuremath{\mathbb N}}

\newcommand{\pss}[2]{\ensuremath{{\langle #1,#2\rangle}}}

\date{\today}
\title[Operators of Read's type and the Hypercyclicity Criterion]{On the Hypercyclicity Criterion for operators of Read's type}

\author{Sophie Grivaux}
\address[S.~Grivaux]{CNRS, Univ. Lille, UMR 8524 - Laboratoire Paul Painlev\'{e}, France}
\email{sophie.grivaux@univ-lille.fr}
\urladdr{http://math.univ-lille1.fr/~grivaux/}
 \subjclass{Invariant Subspace/Subset Problem, operators of Read's type, hypercyclic operators, Hypercyclicity Criterion}
 \thanks{This work was supported in part by
 the project FRONT of the French
National Research Agency (grant ANR-17-CE40-0021) and by the Labex CEMPI (ANR-11-LABX-0007-01)}

\begin{abstract}
 Let $T$ be a so-called \emph{operator of Read's type} on a (real or complex) separable Banach space, having no non-trivial invariant subset. We prove in this note that $T\oplus T$ is then hypercyclic, \emph{i.e.}\ that $T$ satisfies the Hypercyclicity Criterion.
\end{abstract}

\begin{document}
\maketitle
\section{The Invariant Subspace Problem}
Given a (real or complex) infinite-dimensional separable Banach space $X$, the Invariant Subspace Problem for $X$ asks whether every bounded operator $T$ on $X$ admits a non-trivial invariant subspace, \emph{i.e.}\ a closed subspace $M$ of $X$ with $M\neq\{0\}$ and $M\neq X$ such that $T(M)\subseteq M$. It was answered in the negative in the 80's, first by Enflo \cite{E} and then by Read \cite{R1}, who constructed examples of separable Banach spaces supporting operators without non-trivial closed invariant subspace. One of the most famous open questions in modern operator theory is the Hilbertian version of the Invariant Subspace Problem, but it is also widely open in the reflexive setting: to this day, all the known examples of operators without non-trivial invariant subspace live on non-reflexive Banach spaces.

Read provided several classes of operators on $\ell_{1}(\N)$ having no non-trivial invariant subspace \cite{R2}, \cite{R3}, \cite{R4}. In the work \cite{R5}, he gave examples of such operators on $c_{0}(\N)$ and $X=\bigoplus_{\ell_{2}}J$, the $\ell_{2}$-sum of countably many copies of the James space $J$; since $J$ is quasi reflexive (\emph{i.e.}\ has codimension $1$ in its bidual $J^{**}$), the space $X$ has the property that $X^{**}/X$ is separable. This approach was further developed in \cite{RR2}, where it was shown that whenever $Z$ is a non-reflexive separable Banach space admitting a Schauder basis, the $\ell_{p}$-sums $X=\bigoplus_{\ell_{p}}Z$ of countably many copies of $Z$ ($1\le p<+\infty$) as well as the $c_{0}$-sum $X=\bigoplus _{c_{0}}Z$ support an operator without non-trivial invariant closed subspace. Actually, these spaces support an operator without non-trivial invariant closed \emph{subset}. This generalizes a result of Read, who exhibited in \cite{R6} the first known example of an operator (on the space $\ell_1(\N)$) without
non-trivial invariant closed subset. The most recent counterexample to the Invariant Subspace Problem is given in the joint
work by Gallardo-Gutti\'{e}rez and Read \cite{GR}, which happens to be Read's last
article: the authors give an example of a quasinilpotent operator $T$ on $\ell_{1}(\N)$ with the property that whenever $f$ is the germ of a holomorphic function at $0$, the operator $f(T)$ has no non-trivial invariant closed subspace.
\par\medskip
On the other hand, many powerful techniques have been developed in the past decade to show that operators enjoying certain additional properties have non-trivial invariant subspaces. Among these, some of the most interesting have been developed by Lomonosov: his best-known result in this direction, striking for its simplicity and effectiveness, states that every operator on a Banach space commuting with a  compact operator admits a non-trivial invariant subspace \cite{L1}. Another important work of Lomonosov concerns the generalizations of the Burnside inequality obtained in \cite{L2} and \cite{L3} (see \cite{LS} for a simpler proof, relying on nonlinear arguments from \cite{L1}). The Lomonosov inequality from \cite{L2} runs as follows: 
\par\medskip
\textbf{The Lomonosov inequality.}
Let $X$ be a complex separable Banach space, and let $\mathcal{A}$ be a weakly closed subalgebra of $\mathcal{B}(X)$ with $\mathcal{A}\neq\mathcal{B}(X)$. There exist two non-zero elements $x^{*}$ and $x^{**}$ of $X^{*}$ and $X^{**}$ respectively such that $|\pss{x^{**}}{A^{*}x^{*}}|\le ||A||_{e}$ for every $A\in\mathcal{A}$.
\par\medskip
\noindent
Here $||A||_{e}$ denotes the essential norm of $A$, which is the distance of $A$ to the space of compact operators on $X$. 

This inequality is a powerful tool and has been used in many contexts to prove the existence of non-trivial invariant subsets or subspaces for certain classes of operators (see for instance \cite{AGK}, \cite{P1}, \cite{P2}, \cite{RR2}). It is one of the main results which supports the conjecture that adjoint operators on infinite-dimensional dual Banach spaces have non-trivial invariant subspaces.
\par\medskip
It would be impossible to mention here all the beautiful existence results for invariant subspaces proved in the past decade. We refer to the books \cite{RaRo} and \cite{CP} for a description of many of these. We conclude this introduction by mentioning the important work \cite{AH} of Argyros and Haydon, who constructed an example of a space $X$ on which any  operator is the sum of a multiple of the identity and a compact operator. As a consequence of the Lomonosov Theorem \cite{L1}, every operator on $X$ has a non-trivial invariant subspace. Subsequent work of Argyros and Motakis \cite{AM} shows the existence of reflexive separable Banach spaces on which any operator has a non-trivial invariant subspace. Again, the Lomonosov Theorem is brought to use in the proof, although the spaces of \cite{AM} do support operators which are not the sum of a multiple of the identity and a compact operator.

\section{Hypercyclic operators and the Hypercyclicity Criterion}
Let us now shift our point of view, and consider the Invariant Subspace and Subset Problems from the point of view of orbit behavior. It is not difficult to see that $T\in\mathcal{B}(X)$ has no non-trivial invariant subspace if and only if every non-zero vector $x\in X$ is \emph{cyclic} for $T$: the linear span in $X$ of the orbit $\{T^{n}x\,;\,n\ge 0\}$ of the vector $x$ under the action of $T$ is dense in $X$. In a similar way, $T$ has no non-trivial invariant closed subset if and only if every vector $x\neq 0$ is \emph{hypercyclic}, \emph{i.e.}\ the orbit $\{T^{n}x\,;\,n\ge 0\}$ itself is dense in $X$. An operator is called \emph{hypercyclic} if it admits a hypercyclic vector (in which case it admits a dense $G_{\delta }$ set of such vectors).
\par\medskip
The study of hypercyclicity and related notions fits into the framework of \emph{linear dynamics}, which is the study of the dynamical systems given by the action of a bounded operator on a separable Banach space. It has been the object of many investigations in the past years, as testified by the two books \cite{GEP} and \cite{BM} which retrace important recent developments in this direction. One of the main open problems  in hypercyclicity theory was solved in 2006  by De la Rosa and Read \cite{DR}. They constructed an example of a hypercyclic operator $T$ on a Banach space $X$ such that the direct sum $T\oplus T$ of $T$ with itself on $X\oplus X$ is not hypercyclic. In other words, although there exists $x\in X$ with the property that for every 
$u\in X$ and every $\varepsilon >0$, there exists $n\ge 0$ such that $||T^{n}x-u||<\varepsilon $, there is no pair $(x,y)$ of vectors of $X$ such that for every $(u,v)\in X\times X$ and every $\varepsilon >0$, there exists $n\ge 0$ which simultaneously satisfies $||T^{n}x-u||<\varepsilon $ and $||T^{n}y-v||<\varepsilon $. Further examples of such operators (hypercyclic but not \emph{topologically weakly mixing}) were constructed by Bayart and Matheron in \cite{BM0} on many classical spaces such as the spaces $\ell_{p}(\N)$, $1\le p<+\infty$ and $c_{0}(\N)$.
\par\medskip
The question of the existence of hypercyclic operators $T$ such that $T\oplus T$ is not hypercyclic  arose in connection with the so-called Hypercyclicity Criterion, which is certainly the most effective tool for proving that a given operator is hypercyclic. Despite its somewhat intricate form, which we recall below, it is very easy to use.
\par\medskip

\textbf{The Hypercyclicity Criterion.} Let $T\in\mathcal{B}(X)$. Suppose that there exist two dense subsets $D$ and $D'$ of $X$, a strictly increasing sequence $(n_{k})_{k\ge 0}$ of integers, and a sequence $(S_{n_{k}})_{k\ge 0}$ of maps from $D'$ into $X$ satisfying the following three assumptions:
\begin{enumerate}
 \item [(i)] ${T^{n_{k}}x}\rightarrow{0}$ as ${k}\rightarrow{+\infty}$ for every $x\in D$;
 \item[(ii)] ${S_{n_{k}}y}\rightarrow{0}$ as ${k}\rightarrow{+\infty}$ for every $y\in D'$;
 \item[(iii)] ${T^{n_{k}}S_{n_{k}}y}\rightarrow{y}$ as ${k}\rightarrow{+\infty}$ for every $y\in D'$.
\end{enumerate}
Then $T$ is hypercyclic, as well as $T\oplus T$.

\par\medskip
The Hypercyclicity Criterion admits many equivalent formulations, which we will not detail here. An important result, due to B\`{e}s and Peris \cite{BePe}, shows that $T$ satisfies the Hypercyclicity Criterion if and only if $T\oplus T$ is hypercyclic. This criterion is thus deeper than one may think at first glance. Many sufficient conditions implying the Hypercyclicity Criterion have been proved over the years, always in the spirit that ``hypercyclicity plus some regularity assumption implies the Hypercyclicity Criterion'', see \cite[Ch.\,3]{GEP}. For instance, hypercyclicity plus the existence of a dense set of vectors with bounded orbit implies that the Hypercyclicity Criterion is satisfied (\cite{G}, see also \cite[Sec.\,5]{GMM} for generalizations). This phenomenon is well-known in dynamics: an irregular behavior of some orbits (density) combined with the regular behavior of some other orbits (typically, periodicity) implies chaos. See for instance \cite{BBCDS}.

\section{Operators without non-trivial invariant subsets and the Hypercyclicity Criterion}
In the light of this observation (and also of the fact that Read had a hand in the construction of operators without non-trivial invariant subsets, as well as in the construction of hypercyclic operators which are not weakly topologically mixing!), the following question comes naturally to mind:

\begin{question}\label{Question 1}
 Does there exist a bounded operator $T$ on a Banach space $X$ which simultaneously satisfies
 \begin{enumerate}
  \item [(a)] $T$ has no non-trivial invariant subset, that is, all non-zero vectors $x\in X$ are hypercyclic for $T$;
  \item[(b)] $T\oplus T$ is not hypercyclic as an operator on $X\oplus X$?
 \end{enumerate}
\end{question}

 One may be tempted to guess that operators whose set of hypercyclic vectors is too large are somehow less likely to satisfy the Hypercyclicity Criterion
than others (since the usual regularity assumptions may be missing), or one may be inclined to believe that such operators should  indeed satisfy the Criterion (as the set of hypercyclic vectors is so large, there is every chance that there exists a pair $(x,y)$ of vectors of $X$ whose orbits are independent enough for $x\oplus y$ to have a dense orbit under the action of $T\oplus T$). Both arguments are plausible, and it is difficult to get a deeper intuition in Question \ref{Question 1}, besides saying that it is probably hard!
\par\medskip

Our aim in this note is to prove the following modest result, which shows that all the known examples of operators  without non-trivial invariant  closed subset do satisfy the Hypercyclicity Criterion. 

\begin{theorem}\label{Th 2}
 Let $T$ be an operator of Read's type, acting on a (real or complex) separable Banach space, and having no non-trivial invariant subset. Then $T\oplus T$ is hypercyclic, \emph{i.e.}\ $T$ satisfies the Hypercyclicity Criterion.
\end{theorem}

What are \emph{operators of Read's type}? We group under this rather vague denomination all the operators which satisfy certain structure properties, appearing in the constructions carried out by Read, and common to almost all the operators which have no (or few) non-trivial invariant subspaces or subsets. All the operators constructed by Read in \cite{R1, R2, R3, R4, R5, R6}, as well as the operators from \cite{RR1} and \cite{RR2}, fall within this category  (Enflo's examples are of a different type). See \cite[Sec.\,2]{RR2} for an informal description of the properties of operators of Read's type. As will be seen in Section \ref{Sec 4} below, only two of the properties of operators of Read's type are involved in the proof of Theorem \ref{Th 2}, so that it could potentially be applied to much wider classes of operators.

\section{Proof of Theorem \ref{Th 2}}\label{Sec 4}
We will carry out this proof in the context of \cite{RR2}, and will in particular use the notation introduced in \cite[Sec.\,2.2]{RR2}. Read's type constructions involve two sequences $(f_{j})_{j\ge 0}$ and $(e_{j})_{j\ge 0}$ of vectors, defined inductively. The sequence $(f_{j})_{j\ge 0}$ is a Schauder basis of the space $X$. When $X$ is a classical space like $\ell_{1}(\N)$ or $c_{0}(\N)$, $(f_{j})_{j\ge 0}$ is simply the canonical basis of $X$. The vectors $e_{j}$, $j\ge 0$, are defined in such a way that  $e_{0}=f_{0}$ and $\textrm{span}\,[e_{0},\dots,e_{j}]=\textrm{span}\,[f_{0},\dots,f_{j}]$ for every $j\ge 1$. They are thus linearly independent and span a dense subspace of $X$. The operator $T$ is then defined by setting $Te_{j}=e_{j+1}$ for every $j\ge 0$; this definition makes sense since the vectors $e_{j}$ are linearly independent. The whole difficulty of the construction is to define the vectors $e_{j}$ in such a way that $T$ extends to a bounded operator on $X$, and that $T$ has no non-trivial invariant subspace (or subset). Observe that $T^{j}e_{0}=e_{j}$ for every $j\ge 0$, \emph{i.e.}\ that $(e_{j})_{j\ge 0}$ is the orbit of $e_{0}$ under the action of $T$. In particular, $e_{0}$ is by construction a cyclic vector for $T$.
\par\medskip
The vectors $e_{j}$ are defined differently, depending on whether $j$ belongs to what is called in \cite{RR1} or \cite{RR2} a \emph{working interval} or a \emph{lay-off interval}. Lay-off intervals lie between the working intervals, and if $I=[\nu+1,\nu+l]$ is such a lay-off interval of length $l$, $e_{j}$ is defined for $j\in I$ as 
\[
e_{j}=2^{-\frac{1}{\sqrt{l}}(\frac{l}{2}+\nu+1-j)}f_{j}
\]
and $Tf_j=2^{-\frac{1}{\sqrt{l}}}f_{j+1}$  for every $\nu+1\le j<\nu+l$.
\par\medskip
The working intervals are of three types: (a), (b), and (c).
The (c)-working intervals appear only in the case where one is interested in constructing operators without non-trivial invariant subset. These are the only working intervals which will be relevant here. One of their roles is to ensure that $e_{0}$ is not only cyclic, but hypercyclic for $T$. There is at each step $n$ of the construction a whole family of (c)-working intervals, which is called in \cite{RR1} and \cite{RR2} the (c)-fan. The first of these intervals has the form $[c_{1,n},c_{1,n}+\nu _{n}]$, where $\nu _{n}$ is the index corresponding to the end of the last (b)-working interval constructed at step $n$, and $c_{1,n}$ is extremely large with respect to $\nu _{n}$. In order to simplify the notation, we set $c_{n}=c_{1,n}$ for every $n\ge 0$. Thus $[\nu _{n}+1,c_{n}-1]$ is the lay-off interval which precedes the first (c)-working interval. For $j\in[c_{n},c_{n}+\nu _{n}]$, the vector $e_{j}$ is defined as
\[
e_{j}=\gamma _{n}f_{j}+p_{n}(T)e_{0}
\]
where $\gamma _{n}>0$ is extremely small and $p_{n}$ is a polynomial with suitably controlled degree, and such that $|p_{n}|\le 2$ (the polynomial $p_{n}$ is denoted by $p_{1,n}$ in \cite{RR1} and \cite{RR2}; again we simplify the notation). Here the modulus $|p|$ of a polynomial $p$ is defined as the sum of the moduli of its coefficients. 

Thus, in particular, $e_{c_{n}}=T^{c_{n}}e_{0}=\gamma _{n}f_{c_{n}}+p_{n}(T)e_{0}$ and 
$||e_{c_{n}}-p_{n}(T)e_{0}||=\gamma _{n}$. The family $(p_{n})_{n\ge 1}$ is chosen in such a way that for every polynomial $p$ with $|p|\le 2$ and every $\varepsilon >0$, there exists $n\ge 1$ such that $||p_{n}(T)e_{0}-p(T)e_{0}||<\varepsilon $. Hence there exists for every polynomial $p$ with $|p|\le 2$ and every $\varepsilon >0$ an integer $n$ such that $||T^{c_{n}}e_{0}-p(T)e_{0}||<\varepsilon $.
\par\medskip
An important observation is that this property actually extends to \emph{all} polynomials $p$, regardless of the size of their moduli $|p|$. The simple argument is given already in the proof of \cite[Th.\,1.1]{RR1} and in \cite[Sec.\,3.1]{RR2}, but we recall it briefly for the sake of completeness: let $p$ be any polynomial, and fix $\varepsilon >0$. Let $j$ be an integer such that $|p|\le 2^{j}$. Then we know that there exists an integer $n_{1}$ such that $||T^{c_{n_{1}}}e_{0}-2^{-j}p(T)e_{0}||<\varepsilon 2^{-2j}$. There also exists an integer $n_{2}$ such that
$||T^{c_{n_{2}}}e_{0}-2T^{c_{n_{1}}}e_{0}||<\varepsilon 2^{-(2j-1)}$. Then it follows that 
$||T^{c_{n_{2}}}e_{0}-2^{-(j-1)}p(T)e_{0}||<\varepsilon 2^{-2(j-1)}$. Continuing in this fashion, we obtain that there exists an integer $n_{j}$ such that $||T^{c_{n_{j}}}e_{0}-p(T)e_{0}||<\varepsilon $, which proves our claim: there exists for every polynomial $p$ and every $\varepsilon >0$ an integer $n$ such that $||T^{c_{n}}e_{0}-p(T)e_{0}||<\varepsilon $.
\par\medskip
Before moving over to the proof of Theorem \ref{Th 2}, we recall the following result from \cite{G}, which provides a useful sufficient condition for the Hypercyclicity Criterion to be satisfied:
\begin{theorem}[\cite{G}]\label{Th 3}
Let $T$ be a bounded operator on a separable Banach space $X$. Suppose that for every pair $(U,V)$ of non-empty open subsets of $X$, and for every neighborhood $W$ of $0$, there exists a polynomial $p$ such that  $p(T)(U)\cap W$ and $p(T)(W)\cap V$ are simultaneously non-empty. If $T$ is hypercyclic, then $T$ satisfies in fact the Hypercyclicity Criterion.
\end{theorem}

Theorem \ref{Th 3} can be rewritten in somewhat more concrete terms as:
\begin{proposition}\label{Prop 4}
 Let $T$ be a hypercyclic operator on a separable Banach space $X$, and let  $x_{0}$ be a cyclic vector for $T$. If there exist a sequence $(q_{k})_{k\ge 0}$ of polynomials and a sequence $(w_{k})_{k\ge 0 }$ of vectors of $X$ such that
 \[
{q_{k}(T)x_{0}}\rightarrow{0,}\quad {w_{k}}\rightarrow{0,}\quad \textrm{and}\quad 
{q_{k}(T)w_{k}}\rightarrow{x_{0}}
\]
as ${k}\rightarrow{+\infty,}$ then $T$ satisfies the Hypercyclicity Criterion.
\end{proposition}

We are now ready for the proof of Theorem \ref{Th 2}.

\begin{proof}[Proof of Theorem \ref{Th 2}]
Let $(n_{k})_{k\ge 0}$ be a strictly increasing sequence of integers such that 
\[
||T^{c_{n_{k}}}e_{0}-4^{k}e_{0}||<1\quad \textrm{for every } k\ge 1.
\]
 Write $c_{n_{k}}$ as $c_{n_{k}}=i_{n_{k}}+j_{n_{k}}$ where $i_{n_{k}}=\lfloor c_{n_{k}}/2\rfloor$ and $j_{n_{k}}=c_{n_{k}}-\lfloor c_{n_{k}}/2\rfloor$.
 \par\medskip
 Since $c_{n}$ is extremely large with respect to $\nu _{n}$ at each step $n$ of the construction of $T$, $i_{n_{k}}$ belongs to the lay-off interval $[\nu _{n_{k}}+1,c_{n_{k}}-1]$ for every $k$. Thus
 \[
||T^{i_{n_{k}}}e_{0}||=||e_{i_{n_{k}}}||=2^{-\frac{1}{\sqrt{c_{n_{k}}-\nu _{n_{k}}-1}}\bigl(\frac{1}{2}(c_{n_{k}}-\nu _{n_{k}}-1)+\nu_{n_{k}}+1-i_{n_{k}}\bigr)}
\]
and ${||T^{i_{n_{k}}}e_{0}||}\rightarrow{1}$ as ${k}\rightarrow{+\infty.}$ Exactly the same argument shows that ${||T^{j_{n_{k}}}e_{0}||}\rightarrow{1}$ as ${k}\rightarrow{+\infty.}$ 

Set $w_{k}=2^{-k}T^{i_{n_{k}}}e_{0}$ and $q_k(T)=2^{-k}T^{j_{n_{k}}}$ for every $k\ge 1$. Then ${w_{k}}\rightarrow{0}$ and $q_k(T)e_{0}\rightarrow{0.}$  Moreover,
$q_k(T)w_{k}=4^{-k}T^{c_{n_{k}}}e_{0}\rightarrow{e_{0}.}$
 Since the vector $e_{0}$ is hypercyclic for $T$, the assumptions of Proposition \ref{Prop 4} are in force, and $T$ satisfies the Hypercyclicity Criterion.
\end{proof}

\begin{remark}
The same argument shows that the hypercyclic operators from \cite{RR1}, which have few non-trivial invariant subsets but still do have some non-trivial invariant subspaces, also satisfy the Hypercyclicity Criterion. The fact that the operators of Read's type on a separable infinite-dimensional complex Hilbert space $H$ from \cite{RR1} have a non-trivial invariant subspace relies on the Lomonosov inequality from \cite{L2}:
there exists a pair $(x,y)$ of non-zero vectors of $H$ such that $|\pss{T^n x}{y}|\le ||T^n||_e$ for every integer $n$. Since the operators $T$ are by construction compact perturbations of power-bounded (forward) weighted shifts with respect to a fixed Hilbertian basis $(f_j)_{j\ge 0}$ of $H$,   $\sup_{n\ge 0}|\pss{T^n x}{y}|<+\infty$, and the closure of the orbit of $x$ under the action of $T$ is a non-trivial closed invariant subset for $T$.
Moreover, the operator $T$ has the following
property (called (P1) in \cite{RR2}): all closed invariant subsets of $T$ are
actually closed invariant subspaces. Therefore, $T$ has a non-trivial closed invariant subspace. See \cite[Sec.\,7.2]{RR2} for details and more general results.
\end{remark}

We conclude this note with the following question, which may help to shed a light on Question \ref{Question 1}:

\begin{question}\label{Question 2}
Let $T$ be one of the operators from \cite{DR} or \cite{BM0} which are hypercyclic but do not satisfy the Hypercyclicity Criterion. What can be said about the size of the set $HC(T)$ of hypercyclic vectors for $T$? Is it ``large'', or rather ``small''? Is its complement Haar-null, for instance?  
\end{question}

\par\bigskip
\textbf{Acknowledgement:} I am grateful to Gilles Godefroy and Quentin Menet for interesting comments on a first version of this note.

\end{document}